\documentclass{amsart}

\DeclareMathSymbol{\twoheadrightarrow}  {\mathrel}{AMSa}{"10}

\def\P{{\mathbb P}}

                     \def\fFF{{\mathfrak f}}

\def\A8{{\mathbf A}_8}
\def\Bir{\mathrm{Bir}}

\def\Aut{\mathrm{Aut}}

\def\A{\mathcal{A}}

\def\T{{\mathcal T}}

                                       \def\Diff{\mathrm{Diff}}
\def\dim{\mathrm{dim}}

       \def\PGL{\mathrm{PGL}}

          \def\l1{{\mathbf 1}}

\newtheorem{thm}{Theorem}[section]
\newtheorem{lem}[thm]{Lemma}
\newtheorem{cor}[thm]{Corollary}

\theoremstyle{definition}
\newtheorem{defn}[thm]{Definition}

\newtheorem{rem}[thm]{Remark}

        \newtheorem{sect}[thm]{}
\title[Jordan groups]{Jordan groups and algebraic surfaces}
\thanks{The first named author is partially supported by the Ministry
of Absorption (Israel), the Israeli Science Foundation (Israeli
Academy of Sciences, Center of Excellence Program), the Minerva
Foundation (Emmy Noether Research Institute of Mathematics). }
\thanks{The second named author is partially supported by a grant from the Simons Foundation (\#246625 to Yuri Zarkhin).}

\author {Tatiana Bandman}

\author[Yuri G.\ Zarhin]{Yuri G.\ Zarhin}
\address{Department of
Mathematics, Bar-Ilan University, 5290002, Ramat Gan, ISRAEL}
\email{bandman@macs.biu.ac.il}
\address{Department of Mathematics, Pennsylvania State University,
University Park, PA 16802, USA}
\address{Department of Mathematics, The Weizmann Institute of Science,
 POB 26,  Rehovot 7610001, Israel}

\email{zarhin\char`\@math.psu.edu}

\begin{document}

\begin{abstract}
We prove that an analogue of Jordan's theorem on finite subgroups of general linear groups  holds
 for the groups of biregular automorphisms of  algebraic surfaces. This gives a positive answer to a question of Vladimir L. Popov.
\end{abstract}


\maketitle

\section{Introduction}

Throughout this paper,  $k$ is an algebraically closed field of characteristic zero and  $\P^1$ is the projective line  over $k$.
Let $U$ be an  {\sl algebraic variety} over $k$ \cite[Vol. 2, Ch. VI, Sect. 1]{ShAG}.  Then $U(k)$ and  $\Aut(U)$  stand for its set of $k$-points and  the group of biregular $k$-automorphisms  respectively. Unless otherwise stated, by a point of $U$ we mean a $k$-point. If $U$ is {\sl irreducible} then we write  $k(U)$ and $\Bir(U)$ for its field of rational functions and the  group of  birational
$k$-automorphisms respectively; $\Aut(U)$ is a subgroup of $\Bir(U)$.
By an elliptic curve we mean an irreducible smooth projective curve of genus 1 over $k$.
  If $X$ is an elliptic curve and $\T \subset X(k)$ is a {\sl nonempty} finite set of points on $X$ then the (sub)group
 $$\Aut(X,\T) =\{u \in \Aut(X)\mid u(\T)=\T\} \subset \Aut(X)$$
  is {\sl finite}, since $X \setminus \T$ is a {\sl hyperbolic} curve.
If $\mathcal{S}$ is a smooth irreducible projective surface over $k$ then an irreducible closed curve $C$ in $\mathcal{S}$ is called a $(-1)$-curve if it is smooth rational and its self-intersection index is $-1$.

The following definition was inspired by the classical theorem of Jordan \cite[Sect. 36]{CR} about finite subgroups of general linear groups over fields of characteristic zero.

\begin{defn}[Definition 2.1 of \cite{Popov}]
A group $B$ is called a {\sl Jordan group} if there exists a positive integer $J_B$ such that every finite subgroup $B_1$ of $B$ contains a
normal commutative subgroup, whose index in $B_1$ is at most $J_B$.
\end{defn}

\begin{rem}
\label{Jsub}
Clearly,  a subgroup of a Jordan group is also Jordan. If a Jordan group  $G_1$ is a subgroup of {\sl finite} index in a group $G$ then $G$ is also Jordan.
\end{rem}

 V. L. Popov (\cite[Sect. 2]{Popov}, see also \cite{Popov2}) posed a question whether $\Aut(S)$ is a Jordan group when $S$ is an  algebraic surface over $k$.
 He obtained a positive answer to his question for almost all
 surfaces. (The case of rational surfaces was treated earlier by J.-P. Serre \cite[Sect. 5.4]{Serre}). The only remaining case is  when $S$ is birationally (but not biregularly) isomorphic to a product $X\times\P^1$ of an elliptic curve $X$ and the projective line.
In \cite{ZarhinSurface} the second named author proved that $\Aut(S)$ is a Jordan group if $S$ is a {\sl projective} surface. The aim of this paper is to extend this result to the case of arbitrary algebraic surfaces.
 Our main result is the following statement, which gives a positive  answer to Popov's question.

 \begin{thm}
 \label{elliptic}
 If $X$ is an elliptic curve over $k$  and $S$ is an irreducible normal  algebraic surface that is birationally isomorphic to $X\times\P^1$ then
 $\Aut(S)$ is a Jordan group.
 \end{thm}

\begin{rem}
 The group $\Bir(X\times\P^1)$  is {\sl not}  Jordan  \cite{ZarhinEdinburgh}.
 \end{rem}

 \begin{rem}
\label{quasiP}
 Suppose that $S$ is a {\sl non}-smooth irreducible normal surface. Since it is normal, there are only finitely many  singular points on $S$. Then, by \cite[Sect. 2, Cor. 8]{Popov2}, $\Aut(S)$ is Jordan.
 This implies that in the course of the proof of Theorem \ref{elliptic} we may assume that $S$ is smooth. On the other hand,
by a theorem of Zariski \cite[Cor. II.2.6 on p. 53]{Zar}, every irreducible smooth surface is quasi-projective. This implies that in the course of the proof of Theorem \ref{elliptic} we may assume that $S$ is {\sl smooth quasi-projective}.
 \end{rem}

 \begin{cor}
 \label{surfaceJ}
 Suppose that $V$ is an irreducible normal  algebraic variety over $k$. If $\dim(V) \le 2$ then $\Aut(V)$ is Jordan.
 \end{cor}

 \begin{proof}[Proof of Corollary \ref{surfaceJ}]
 We have $\Aut(V) \subset \Bir(V)$. If $V$ is {\sl not} birationally isomorphic to a product of the projective line and an elliptic curve then
 $\Bir(V)$ is Jordan (\cite[Th. 2.32]{Popov}) and therefore its subgroup $\Aut(V)$ is also Jordan. If $V$ is birationally isomorphic to a product of the projective line and an elliptic curve then $\dim(V)=2$ and Theorem \ref{elliptic}  implies that $\Aut(V)$ is Jordan.
 \end{proof}

\begin{thm}
\label{NonNormal}
Let $V$ be an irreducible  algebraic variety over $k$. If $\dim(V) \le 2$ then $\Aut(V)$ is Jordan.
\end{thm}

\begin{proof}[Proof of Theorem \ref{NonNormal}]
Let $\nu:V^{\nu} \to V$ be the {\sl normalization} of $V$
(\cite[Ch. III, Sect. 8]{MumfordRed}, \cite[Ch. 2,  Sect. 2.14]{Itaka}). Here $\nu$ is a birational (surjective) regular map that is called the normalization map for $V$ and
 $V^{\nu}$ is an irreducible  {\sl normal} variety (of the same dimension as $V$) over $k$ \cite[Th. 4 on p. 203]{MumfordRed}.
 The {\sl universality property} of the normalization map
  implies that every biregular automorphism of $V$ lifts uniquely to a biregular automorphism of $V^{\nu}$ \cite[Ch. 2,  Sect. 2.14, Th. 2.25 on p, 141]{Itaka}. This
give rise to the {\sl embedding} of groups
$$\Aut(V)\hookrightarrow \Aut(V^{\nu}).$$
By  Corollary \ref{surfaceJ}, the group $\Aut(V^{\nu})$ is Jordan. Since $\Aut(V)$ is isomorphic to a subgroup of  Jordan   $\Aut(V^{\nu})$, it is also Jordan.

\end{proof}

\begin{cor}
Let $V$ be an algebraic variety over $k$.
 If $\dim(V)\le 2$ then $\Aut(V)$ is Jordan.
\end{cor}
\begin{proof}
Let $V_1, \dots , V_r$ be all the {\sl irreducible} components of $V$. Clearly, all $V_i$ are irreducible algebraic varieties with
$\dim(V_i) \le \dim(V) \le 2$. By Theorem  \ref{NonNormal}, all $\Aut(V_i)$ are Jordan. Now Lemma 1 in Section 2.2 of \cite{Popov2} implies that $\Aut(V)$ is also Jordan.
\end{proof}

\begin{rem}
Suppose that $k$ is the field $\mathbb{C}$ of complex numbers and $X$ is a smooth irreducible quasi-projective non-projective surface. Then $M=X(\mathbb{C})$ carries the natural structure of a connected oriented smooth {\sl real
noncompact} fourfold and the group $\Aut(X)$ embeds naturally in the group $\Diff(M)$ of the (real)  diffeomorphisms of the fourfold  $M$. While $\Aut(X)$ is always Jordan, there are examples of connected oriented smooth {\sl noncompact} real fourfolds, whose group of diffeomorphisms is {\sl not} Jordan \cite{Popov3}.
\end{rem}

The paper is organized as follows. In Section \ref{Bir1} we discuss {\sl minimal closures} of surfaces.  In Section \ref{mainP} we prove Theorem \ref{elliptic}.

{\bf Acknowledgements}. We are deeply grateful to Vladimir  Popov for a stimulating question and useful discussions.
This work was started in September 2013 when both authors were visitors at the  Max-Planck-Institut f\"ur Mathematik (Bonn), whose hospitality and support are gratefully acknowledged. Most of this work was done during the academic year 2013/2014 when the second named author (Y.Z.) was Erna and Jakob Michael Visiting Professor in the Department of Mathematics  at the Weizmann Institute of Science, whose hospitality and support are gratefully acknowledged.

\section{Minimal closures}
\label{Bir1}

\begin{sect}
 \label{prelim}
 Let $X$ be an elliptic curve over $k$ and $S$ be a {\sl smooth} irreducible quasi-projective surface over $k$ that
 is birationally isomorphic to $X \times \P^1$.
There exists an irreducible smooth projective surface $\bar{S}$ such that its certain  Zariski-open subset is biregularly isomorphic to $S$ (further we identify $S$ with this open subset).
Clearly,  the inclusion map $S \subset \bar{S}$ is a birational morphism.
This implies that
$$\Aut(S) \subset \Bir(S) =\Bir(\bar{S})$$
and therefore one may view $\Aut(S)$ as a subgroup of $\Bir(\bar{S})$.
Since $\bar{S}$ is birationally  isomorphic to $S$, it also   birationally  isomorphic to $X\times\P^1$.

 Let us fix a birational isomorphism between $\bar{S}$ and $X\times\P^1$. The projection map $X\times\P^1\to X$ gives rise to a rational map $\bar{\pi}:\bar{S} \to X$ with dense image. Since $\bar{S}$ is smooth and $X$ becomes abelian variety (after a choice of a base point), it follows from a theorem of Weil \cite[Sect. 4.4]{Neron} that $\bar{\pi}$ is regular.
 Since $\bar{S}$ is projective,  $\bar{\pi}: \bar{S} \to X$ is surjective, because its image is closed.

 For each $x \in X(k)$ we write
$\bar{F}_x$ for the effective divisor $\bar{\pi}^{*}(x)$ on $\bar{S}$ that is the pullback (under $\bar{\pi}$) of the divisor $(x)$ on $\bar{S}$.  Clearly, the support of $\bar{F}_x$ coincides with the  curve $\bar{\pi}^{-1}(x)$ on $\bar{S}$.
One say that the fiber of $\bar{\pi}$ over $x$ is {\sl reduced} if all irreducible components of the divisor $\bar{F}_x$ have multiplicity $1$. We say that  the fiber of $\bar{\pi}$ over $x$ is {\sl irreducible} if the curve  $\bar{\pi}^{-1}(x)$ is irreducible;
if this is the case then its multiplicity in  $\bar{F}_x$ is $1$ \cite[Ch. 3,  Sect. 1.4, Lemma 1.4.1(1) on p. 195]{Mi}.

 It is known
\cite[Ch. IV]{ShSurface} that for all but finitely many  $x \in X(k)$ the fiber of $\bar{\pi}$ over $x$ is  irreducible and reduced, and the curve $\bar{\pi}^{-1}(x)$ is smooth (and irreducible). We call such fibers nonsingular and other fibers {\sl singular}.

 If $C$ is a rational curve on $\bar{S}$ then the restriction of $\bar{\pi}$ to $C$ must be a constant map, because every map from a rational curve to an elliptic curve is constant. This implies that $C$ lies in a fiber of $\bar{\pi}$.
(In particular, every $(-1)$-curve on $\bar{S}$ lies in a fiber of $\bar{\pi}$.) This implies that every birational automorphism of $\bar{S}$ is fiberwise \cite[Sect. 13, Th. 2]{ISH}; see Sect. 2.2 below.

However,  if $x \in X(k)$  and the fiber $\bar{\pi}^{-1}(x)$ is singular
then
the corresponding divisor $\bar{F}_x$  enjoys the following properties \cite[Ch. I, Sect. 2.12;  Ch. 3,  Sect. 1.4, Lemma 1.4.1 on p. 195]{Mi}] (see also \cite{Fu}).

\begin{itemize}
\item[(i)]
Each irreducible component of $\bar{F}_x$ is a smooth rational curve (and the corresponding graph is a tree) \cite[Sect. 3]{Fu}.
\item[(ii)]
At least, one of the irreducible components of $\bar{F}_x$ is a $(-1)$-curve \cite[Sect. 4.2]{Fu}.
\item[(iii)]
If one of the irreducible components of $\bar{F}_x$ is a $(-1)$-curve of multiplicity $1$ then there is another irreducible
 $(-1)$-component of  $\bar{F}_x$ (\cite[Sect. 4.2]{Fu}.
\end{itemize}
\end{sect}

\begin{sect}
\label{barS}

If $\sigma \in \Bir(\bar{S})$ then
there is a unique {\sl biregular} automorphism $\fFF(\sigma): X \to X$ such that the composition $\bar{\pi} \sigma$  is a {\sl regular} map that coincides with the composition
$$\fFF(\sigma)\bar{ \pi}: \bar{S} \stackrel{\bar{\pi}}{\to} X \stackrel{\fFF(\sigma)}{\to} X$$
(see, e.g., \cite[Lecture V, Sect. 1.4, p. 99]{PM}).
 Clearly, $\sigma$ sends the fiber $\bar{\pi}^{-1}(x)$ to the fiber $\bar{\pi}^{-1}(\fFF(\sigma)(x))$ for all $x \in X(k)$.
We get a surjective group homomorphism
$$\fFF: \Bir(\bar{S}) \to \Aut(X), \ \sigma \mapsto \fFF(\sigma)$$
 that fits into a short exact sequence
 $$\{1\} \to \Bir_X(\bar{S}) \subset \Bir(\bar{S}) \stackrel{\fFF}{\to} \Aut(X)\to \{1\}$$
 where the subgroup
 $\Bir_X(\bar{S})$ consists of all birational automorphisms $\sigma \in \Bir(\bar{S})$ such that $\bar{\pi}\sigma=\bar{\pi}$
 (i.e. $\sigma$ leaves invariant every fiber of $\bar{\pi}$). In addition, $\Bir_X(\bar{S})$ is isomorphic to the projective linear group $\PGL(2, k(X))$ over the field $k(X)$ of rational functions on $X$ \cite[Lecture V, Sect. 1.4, p. 99]{PM}.
\end{sect}

 \begin{sect}
\label{open}
We write $\pi$ for the composition
$$S \subset \bar{S} \stackrel{\bar{\pi}}{\to} X,$$
i.e., for the restriction of $\pi$ to $S$.
Recall that $\Aut(S) \subset \Bir(\bar{S})$.
Since $S$ is a surface, it  is not contained in a union of finitely many fibers of $\pi$ in $\bar{S}$. This implies that $\pi(S)$ is infinite and therefore is everywhere dense in $X$. It follows from \cite[vol. 1, Ch. 1, Sect. 5, Th. 6]{ShAG} that either $\pi(S)=X$ or the complement
$T_0:=X(k)\setminus \pi(S(k))$ is a finite set and
$$S \subset \pi^{-1}(X\setminus T_0) \subset \bar{S}.$$
If we write $\Aut_X(S)$ for the intersection (in $\Bir(\bar{S})$) of $\Aut(S)$ and  $\Bir_X(\bar{S})$ then we get a short exact sequence
$$\{1\} \to \Aut_X(S) \subset \Aut(S) \stackrel{\fFF}{\to} f( \Aut(S))\to \{1\}$$
where
$$\Aut_X(S)  \subset \Bir_X(\bar{S}), \  \fFF(\Aut(S)) \subset \Aut(X).$$

Similarly to the case of projective surfaces, if $x \in X(k)$ then we write
$F_x$ for the effective divisor $\pi^{*}(x)$ on $S$ that is the pullback (under $\pi$) of the divisor $(x)$ on $S$.  Clearly, the support of $F_x$ coincides with the curve $\pi^{-1}(x)$ on $S$.  It is also clear that the divisor $F_x$ on $S$ is the pullback of the divisor $\bar{F}_x$ on $\bar{S}$ under the (open) inclusion map $S \subset \bar{S}$.
One says that the fiber of $\pi$ over $x$ is {\sl reduced} if all irreducible components of the divisor $F_x$ have multiplicity $1$. We say that  the fiber of $\pi$ over $x$ is irreducible if it is a multiple of a {\sl simple} divisor, i.e., the curve $\bar{\pi}^{-1}(x)$ is irreducible.  Clearly, if  the fiber of $\bar{\pi}$ over $x$ is irreducible (resp. reduced, resp. smooth) then the fiber of $\pi$ over $x$ is irreducible  (resp. reduced, resp. smooth).
On the other hand,  if $\bar{F}_x$ has an irreducible component, say, $\bar{C}$ that appears in $\bar{F}_x$ with multiplicity  $m>1$  and, in addition,  $\bar{C}$ meets $S$ then $C:=\bar{C}\bigcap S$ is an irreducible curve in $S$ that is a component of $F_x$ and  that appears in $F_x$ with the same multiplicity $m$; in particular, the fiber of $\pi$ over $x$ is {\sl not} reduced.  Notice also that if $\bar{C}_1$ and $\bar{C}_2$ are distinct irreducible components of $\bar{F}_x$
that meet $F_x$
 then
$C_1:=\bar{C}_1\bigcap S$ and $C_2:=\bar{C}_2\bigcap S$ are {\sl distinct} irreducible components of $F_x$; in particular, the fiber of $\pi$ over $x$ is {\sl not} irreducible.

It follows from the results about the fibers of $\bar{\pi}$ mentioned in Sect. \ref{prelim}
(see also theorems of Bertini \cite[vol. 1, Ch. 2, Sect.  6.1 and  6.2]{ShAG}
 that either all the fibers of $\pi$ are smooth irreducible reduced or the set $T_1$ of points $x \in \pi(S(k))\subset X(k)$ such that, at least, one of these properties does not hold, is finite.  Clearly,
$$\fFF(\Aut(S))\subset \Aut(X,T_0), \ \fFF(\Aut(S))\subset \Aut(X,T_1).$$
This implies that if either $T_0$ or $T_1$ is {\sl non}-empty then
$\fFF(\Aut(S))$ is a {\sl finite} group and $\Aut_X(\bar{S})$ is a subgroup of {\sl finite index} in $\Aut(S)$.
\end{sect}

 \begin{sect}
\label{degenerate}
It follows from the theorem of Jordan that the projective linear group
 $\PGL(2, k(X))$ is   Jordan \cite{Popov,ZarhinSurface}. Since  $\Bir_X(\bar{S})$ is isomorphic to $\PGL(2,k(X))$ (see Sect. \ref{barS}), it is also a Jordan group. This implies in turn that its subgroup $\Aut_X(S)$ is also Jordan. It follows that
if either $T_0$ or $T_1$ is {\sl non}-empty then  $\Aut(S)$ contains the Jordan subgroup $\Aut_X(S)$ of finite index and therefore is Jordan itself, thanks to Remark \ref{Jsub}.
\end{sect}
In order to handle the case of empty $T_0$ and $T_1$, we need additional ideas.
\begin{defn}
\label{minimal}
 The projective surface $\bar{S}$ is called a (relative) {\sl minimal closure} of $S$ if every $(-1)$-curve on $\bar{S}$ meets $S$. See \cite[Sect. 4.9]{Fu}.
 A minimal closure of $S$  always exists \cite[Prop. 4.10]{Fu}. (Warning: if $\bar{S}$ is a minimal closure then the complement of $S$ in $\bar{S}$ does {\sl not} have to be a divisor!)
\end{defn}

\begin{lem}[Lemma 4.12 of \cite{Fu}]
\label{irred}  Assume that $\pi(S)=X$ and all the fibers of $\pi$ are smooth irreducible and reduced.

If $\bar{S}$ is a minimal closure of $S$ then all the fibers of $\bar{\pi}: \bar{S} \to X$ are irreducible.
\end{lem}

\begin{proof}
Suppose that there exists $x \in X(k)$ such that the fiber of $\bar{\pi}$ over $x$ is not irreducible and therefore is singular.
Then $\bar{F}_x$ contains as an irreducible component a $(-1)$-curve, say $\bar{C}_1$ with multiplicity $m\ge 1$ (Sect. \ref{prelim}). The minimality of $\bar{S}$ implies that $C_1=\bar{C}_1 \bigcap S$ is non-empty and therefore is an irreducible component of $F_x$ with the same multiplicity $m$ (Sect. \ref{open}). Since the fiber of $\pi$ over $x$ is reduced, $m=1$. This implies that $\bar{F}_x$ contains another irreducible component $\bar{C}_2$ that is also
a $(-1)$-curve.  Again $C_2=\bar{C}_2 \bigcap S$  is an irreducible component of $F_x$  that does not coincide with $C_1$. This implies that the fiber of $\pi$ over $x$ is {\sl not} irreducible, which is not the case.
\end{proof}

\begin{thm}
\label{extension}
Assume that $\pi(S)=X$ and all the fibers of $\pi$ are smooth irreducible and reduced. Let  $\bar{S}$ be a minimal closure of $S$  Then every biregular automorphism of $S$ extends uniquely to a biregular automorphism of $\bar{S}$. In other words,
$$\Aut(S) \subset \Aut(\bar{S})\subset \Bir(\bar{S}).$$
\end{thm}

\begin{proof}
 By Lemma \ref{irred}, every fiber
$\bar{F}_x$
is  an irreducible curve isomorphic to $\P^1$.

Let $g: S \to S$ be a biregular automorphism of $S$. Let us extend $g$ to a birational map
$$\bar{g}: \bar{S} \to \bar{S}.$$
 Assume that $\bar{g}$ is {\sl not} a regular map.
 Let   $S'$ be a {\sl resolution of the indeterminacies} of $\bar{g}$, i.e.
a smooth irreducible surface included into the following commutative digram.
\begin{equation}\label{diagram}
\begin{aligned}
& &   & S^{\prime}        &  {}  & {}        &   & & \notag \\
& &u& \downarrow  &        {\searrow}                     &  g^{\prime}  & & &\notag \\
& & &\bar{S}         &\stackrel{\bar{g} }{\dashrightarrow } &\bar{S}   & & &\notag \\
& & &\cup       &  {}                              &\cup       &  & &\notag \\
& & &S              & \stackrel{ g}{\longrightarrow }    & S  & &          &\notag \\
&  & \pi& \downarrow  &                              &  \downarrow & \pi& &\notag \\
& &  &X            &\stackrel {h}{\longrightarrow} &X  &&    &        \end{aligned},  \end{equation}
where $u$ is a  birational morphism that is a composition of finitely many blow ups and induces a biregular isomorphism between $u^{-1}(S)$ and $S$
(such an $u$ exists, because $g$ is defined on $S$),
  $g^{\prime}$ and $\bar{\pi}^{\prime}=\bar{\pi}\circ u$
are
morphisms, and  $h=\fFF(g) \in \Aut(X)$ is a biregular automorphism of $X$. (The group homomorphism $\fFF$ is defined in Sect. \ref{barS}.)
Let $D^{\prime}\subset  S^{\prime}$ be the union of all exceptional curves for $g^{\prime}$ and let $D=g^{\prime}(D^{\prime})\subset \bar{S}$,
which is a finite set.

{\it Every point $z$ of $\bar{S}$ that does {\sl not} lie on $D$ has only one preimage ${g^{\prime}}^{-1}(z)\in S^{\prime}$ (\cite[Ch. 2, Sect. 4, Th. 2]{ShAG}).}

  Let $B^{\prime}$ be the union of exceptional curves for $u$. Clearly,
$$B^{\prime}\subset S^{\prime} \setminus u^{-1}(S).$$
This implies that
$$u(B^{\prime}) \bigcap S = \emptyset .$$
We want to show that $B^{\prime}\subset D^{\prime}$, because then one may contract all
components of $B'$ and  $\bar{g}$  would appear to be a morphism.

Let $C^{\prime}$ be an irreducible component of $B^{\prime}$.
The {\sl point} $u(C^{\prime})$ lies in  $u(B^{\prime})$ and therefore does {\sl not} belong to $S$.

 Since $X$ is an elliptic curve, and $C^{\prime}$ is rational,  $\bar{\pi} (g^{\prime}(C^{\prime}))$ is a point $x\in X(k)$. Thus, since all the  fibers of $\bar{\pi}$ are irreducible  (thanks to Lemma \ref{irred}), either

{\bf Case 1.}    $g^{\prime}(C^{\prime})$ is a point and therefore  $C^{\prime}\subset D^{\prime};$

 or

{\bf Case 2.} $g^{\prime}(C^{\prime})=\bar{F}_x=\bar{\pi}^{-1}(x)\subset \bar{S}$.
Let us put  $x_1:=h^{-1}(x) \in X(k)$. Then $x=h(x_1)\in X(k)$.
Let $s\in F_x\setminus (F_x\cap D)\subset S$ be a point of the fiber $F_x$,  which is not in the image of $D^{\prime}$.  Therefore it has only one {\sl preimage} $s_1:={g^{\prime}}^{-1}(s).$ Moreover, $s_1\in u^{-1}(S)$, because  $s\in S$. On the other hand,
since  $g^{\prime}(C^{\prime})=\bar{F}_x$,
there is a point $c\in C^{\prime}\subset S^{\prime} \setminus u^{-1}(S) $ such that $g^{\prime}(c)=s.$ Clearly, $c \ne s_1$ and we get a contradiction that shows that 
the Case 2 does not occur.

This proves that every $g \in \Aut(S)$ extends to a regular birational map $\bar{g}:\bar{S} \to \bar{S}$. Since the same is true for  $g^{-1} \in \Aut(S)$,  the map $\bar{g}$ is a biregular automorphism of $\bar{S}$.

\end{proof}

\section{Proof of Theorem \ref{elliptic}}
\label{mainP}
Remark \ref{quasiP} tells us that we may assume that $S$ is  a smooth   quasi-projective surface.
In light of results of Section \ref{degenerate}, we may  also assume that  every fiber of $\pi$ is  smooth irreducible and reduced, and $\pi(S)=X$.  Let  $\bar{S}$ be a minimal closure of $S$. By Theorem \ref{extension},
$\Aut(S)$ is a subgroup of $\Aut(\bar{S})$. Since $\bar{S}$ is projective, the results of \cite{ZarhinSurface} imply that the group $\Aut(\bar{S})$ is Jordan and therefore its every subgroup is Jordan. It follows that $\Aut(S)$ is Jordan.

\end{document}